\documentclass[9pt,english]{extarticle}
\usepackage[utf8]{inputenc}
\usepackage{amsmath}
\usepackage{amsfonts}
\usepackage{amssymb}
\usepackage{amsthm}
\usepackage{graphicx}
\usepackage{color}
\usepackage{algorithm}
\usepackage{algorithmicx}
\usepackage{algpseudocode}
\usepackage{caption}
\usepackage{subcaption}
\usepackage{url}
\usepackage{booktabs}
\newcommand{\algrule}[1][.2pt]{\par\vskip.5\baselineskip\hrule height #1\par\vskip.5\baselineskip}

\usepackage{palatino}

\newtheorem{theorem}{Theorem}
\newtheorem{lemma}{Lemma}

\newcommand{\ra}[1]{\renewcommand{\arraystretch}{#1}}
\newcommand{\eps}{\varepsilon}

\newcommand{\Prob}[1]{\ensuremath{\mathbb{P}\left(#1\right)}}

\DeclareMathAlphabet{\mathpzc}{OT1}{pzc}{m}{it}

\newcommand{\OO}{\mathcal{O}}

\newcommand{\RR}{\mathbb{R}}

\newcommand{\norm}[1]{\ensuremath{\left\|#1\right\|_2}}

\newcommand{\frobnorm}[1]{\ensuremath{\left\|#1\right\|_{\text{\rm F}}}}

\newcommand{\nr}[1]{\ensuremath{\mathrm{\textbf{\footnotesize nr}}\left(#1\right)}}

\newcommand{\sr}[1]{\ensuremath{\mathrm{\textbf{\footnotesize sr}}\left(#1\right)}}

\newcommand{\ignore}[1]{}

\newcommand{\mat}[1]{ {\ensuremath{\mathsf{#1} }}}

\newcommand{\matA}{\mat{A}}
\newcommand{\matB}{\mat{B}}

\newcommand{\matG}{\mat{G}}

\newcommand{\matX}{\mat{X}}
\newcommand{\matY}{\mat{Y}}

\newcommand{\matE}{\mat{E}}

\usepackage[left=2cm,right=2cm,top=2cm,bottom=2cm]{geometry}

\usepackage{palatino}

\begin{document}

\title{\Huge{Approximate Matrix Multiplication \\ with Application to Linear Embeddings}}

\author{
  Anastasios Kyrillidis\\
  Computer and Communication Sciences, EPFL \\
  \texttt{anastasios.kyrillidis@epfl.ch}
  \and
  Michail Vlachos\\
   IBM Research Lab, Zurich \\
  \texttt{mvl@zurich.ibm.com}
  \and
  Anastasios Zouzias\thanks{The research leading to these results has received funding from the European Research Council under the European Union's Seventh Framework Programme (FP7/2007-2013) / ERC grant agreement $n^{o}$ 259569.}\\
   IBM Research Lab, Zurich \\
  \texttt{azo@zurich.ibm.com}
}

\maketitle

\begin{abstract}
In this paper, we study the problem of approximately computing the product of two real matrices. In particular, we analyze a dimensionality-reduction-based approximation algorithm due to Sarlos~\cite{sarlos2006improved}, introducing the notion of \emph{nuclear rank} as the ratio of  the nuclear norm over the spectral norm. 
The presented bound has improved dependence with respect to the approximation error (as compared to previous approaches), whereas the subspace -- on which we project the input matrices -- has dimensions proportional to the maximum of their nuclear rank and it is independent of the input dimensions.
In addition, we provide an application of this result to linear low-dimensional embeddings.  Namely, we show that any Euclidean point-set with bounded nuclear rank is amenable to projection onto number of dimensions that is independent of the input dimensionality, while achieving additive error guarantees.
\end{abstract}

\section{Introduction}

Living in the era of \emph{Big Data}, the excess of available information constitutes its manipulation and interpretation a strenuous task: In contrast to conventional wisdom where more data is a source of ``simplicity'' in statistical terms \cite{chandrasekaran2013computational}, large datasets embody preprocessing tasks of high time- and space-complexity, jeopardizing any hope for data analysis within reasonable time and with low computational cost. Due to such difficulties, one might be interested in \emph{screening} the data, even if \emph{no prior information is available}. I.e., identifying a coreset such that most of the latent structure/information that we want to infer is maintaned within an error that we can control. Such approaches have been witnessed in a broad class of data analysis problems such as data matrix sparsification for accelerated spectral calculation \cite{achlioptas2013near}, feature selection (a.k.a. column subset selection) for better interpretation of the results \cite{mahoney2009cur} and low-complexity calculations \cite{BoutsidisMD09}, feature extraction via dimensionality reduction techniques \cite{indyk1998approximate, BoutsidisZD10}, etc. 


In this work, we analyze a particular approximation algorithm of~\cite{sarlos2006improved} for the task of matrix multiplication with respect to their spectral norm \cite{zouzias2013randomized}. Within this context, we focus on approximately computing the product of two matrices \emph{when their intrinsic dimensionality is low, as expressed by the nuclear rank.} Here, we define the nuclear rank as the ratio of the nuclear norm of a matrix over its spectral norm. We provide an elementary proof based on the randomized algorithm described in \cite{sarlos2006improved,zouzias2013randomized}. As a result, we further strengthen the performance of the proposed scheme in situations where the allowed approximation error is small: Using a weaker notion of intrinsic dimensionality (i.e., nuclear rank instead of stable rank; see the definitions later in text), the dependence on the approximation error $\eps$ is improved to $\mathcal{O}(1/\eps^2)$ instead of $\mathcal{O}(1/\eps^4)$, indicating that the proposed scheme scales better when $\eps$ decreases. Table~\ref{table:mm} places our result into context with prior works.
\begin{table}[!h]
\centering
\caption{Summary of results on $\varepsilon$-approximate matrix multiplication $\matA \cdot \matB$ with respect to spectral norm -- $\nr{\cdot}$ and $\sr{\cdot}$ denote the nuclear and stable rank of a matrix, respectively.}\label{table:mm}
\ra{1.3}
\begin{tabular}{l c c c c} \toprule
\textbf{Metric} & \phantom{ab} & \textbf{\# of dimensions} & \phantom{ab} & \textbf{Reference} \\
\cmidrule{1-1} \cmidrule{3-3} \cmidrule{5-5}
 Rank		& & $\OO( ( r(\matA) + r(\matB) ) / \eps^2)$	& & \cite{MZ11}	\\
 Stable rank	& & $\OO( (\sr{\matA} + \sr{\matB}) / \eps^4)$	& 	& \cite{zouzias2013randomized}	\\
 Nuclear norm	&  & $\OO( (\nr{\matA} + \nr{\matB}) / \eps^2)$   &	&   Theorem~\ref{thm:main} \\
\bottomrule
\end{tabular}
\end{table}


As an application, we use this result to design dimensionality reducing linear embeddings that operate on a given dataset and preserve Euclidean distances between data points. In general, it is well-known that simply rotating, scaling and translating in a random way is adequate for this task: According to JL lemma \cite{johnson1984extensions}, such linear mapping approximately preserves (i.e., within a $\eps$-radius of points) the distances between all the pairs of points. While there are many random constructions that achieve this property with high probability, such schemes are oblivious to the geometry of the data set at hand, leaving space for further improvements. 

\vskip.07in
\noindent \textbf{Contributions:} The main contributions of this manuscript are the following:
\begin{enumerate}
\item [$(i)$] We provide a novel analysis on the approximate matrix multiplication problem based on the notion of nuclear rank (See Theorem~\ref{thm:main}).
\item [$(ii)$] We demonstrate an application of the approximate matrix multiplication bound to dimensionality reduction with relative guarantees. 
\end{enumerate}


\section{Related Work}
There is rich literature on the topic of linear dimensionality reduction and the techniques utilized for this purpose, mainly due to the wide range of applications it covers. Here, we highlight a few approaches that provide bounds depending on the input data: These bounds are usually stronger than the classic JL lemma, provided that the input points have low ``geometric complexity'', i.e., points that lie on a low dimensional subspace~\cite{sarlos2006improved}, lie on a manifold~\cite{baraniuk,clarkson2008tighter}, etc. Here, we focus only on results that provide theoretical guarantees.

For the special case where the input points lie on the unit sphere, there is a close connection between the Talagrand functional, denoted as $\gamma_2$, and the required number of dimensions for point distance preservation~\cite{mendelson},\cite{Klartag}. These results can be viewed as stronger bounds compared to the JL lemma. Baraniuk et al. provide bounds on the number of dimensions required by a random linear embedding to preserve Euclidean distances for the case of smooth manifolds~\cite{baraniuk}; for improvements and a nice exposition on the topic, see~\cite{clarkson2008tighter}. Sarlos provided improved bounds for algorithms on large matrices with applications on matrix multiplications, linear regression and low rank matrix approximation~\cite{sarlos2006improved}. On a similar flavor, the authors in~\cite{related:movingPoints} provide bounds for relative error approximations of points that lie on a surface. Several approaches exist for non-linear dimensionality reduction as well~\cite{JL:beyond,JL:beyond2}. From a different perspective, Indyk and Naor consider the ``doubling dimension'' and other related measures of dimensionality for preserving nearest neighbor queries~\cite{nnEmbed}.

Recent developments in \cite{hegde2012convex} describe \emph{deterministic constructions} of linear embeddings in polynomial time via SemiDefinite Programming (SDP) relaxations, under the assumption that the data is known \emph{apriori} and fixed. 
For the latest developments on this topic, we refer the reader to \cite{grant2013nearly}.

\section{Preliminaries}

A scalar is denoted by an italic letter, e.g. $a$. A column vector is denoted by a bold lowercase letter, e.g. $\mathbf{a} \in \mathbb{R}^n$ whose $i$-th entry is $a_i$. A matrix is denoted by a mathtype uppercase letter, e.g., $\matA \in \mathbb{R}^{n \times d}$ with $(i, j)$-th entry $\matA_{i, j}$. To denote the $i$-th row and $j$-th column of $\matA$, we use the $\mathbf{a}_i$ and $\mathbf{a}^{j}$, respectively. Use $\norm{\matA} = \max_{\mathbf{x}: \|\mathbf{x}\|_2 = 1} \|\matA\mathbf{x}\|_2$ to denote the spectral norm of $\matA$ and $\frobnorm{\matA} = \sqrt{\sum_i \sum_j \matA_{i, j}^2}$ to represent its Frobenius norm. 

In our analysis, we utilize the notion of {\it nuclear norm}. Given a matrix $\matA \in \mathbb{R}^{n \times d}$ with rank $r(\matA)$, its nuclear norm $\|\matA\|_{\star}$ is given by $\|\matA\|_{\star} = \sum_{i = 1}^{r(\matA)} \sigma_i$ where $\sigma_i$ is the $i$-th largest singular value; for multiple matrices, we also use $\xi_i$ to denote a singular value. Using the nuclear norm, we define the \emph{nuclear rank} of $\matA$ as $\nr{\matA} := \|\matA\|_{\star} / \norm{\matA}$. Moreover, we use the notion of \emph{stable rank}, which is defined as $\sr{\matA} = \frobnorm{\matA}^2 / \norm{\matA}^2$. 
Throughout the paper, we will denote by $\matG$ a $t \times d$ random matrix whose entries are independent Gaussian random variables with variance $1/t$, i.e., $\matG_{ij} \sim \mathcal{N}(0, 1/t)$.

\section{Main result}
We consider the following problem:
\vskip.1in
\noindent \textsc{Approximate Matrix Multiplication Problem:} Let $\matX \in \RR^{n \times d}$ and $\matY \in \RR^{d \times m}$ be two arbitrary matrices, $\eps > 0$ is an approximation parameter and $0<\delta<1$ the failure probability. We desire to construct \emph{sketched} matrices $\widehat{\matX} \in \RR^{n \times t}$ and $\widehat{\matY} \in \RR^{t \times m}$ where $t \ll d$ such that:
\begin{align}\nonumber
\norm{\widehat{\matX} \widehat{\matY} - \matX \matY} \leq  \eps \norm{\matX} \cdot \norm{\matY},
\end{align} holds with probability at least $1-\delta$.
\vskip.1in
We now state the main theorem of the paper:
\begin{theorem}{\label{thm:main}}
Fix $0 < \eps, \delta < 1$ and assume arbitrary matrices $\matX \in \RR^{n \times d}$ and $\matY \in \RR^{d \times m}$. Set $\widehat{\matX} = \matX \matG^\top$ and $\widehat{\matY} = \matG \matY$. If $t = \Omega \left( \frac{\nr{\matX} + \nr{\matY} + \log(\log(1/\eps)) +\log(1/\delta)} {\eps^2} \right)$, then the following holds:
\begin{equation}{\label{eq:thm}}
\Prob{ \norm{\widehat{\matX} \widehat{\matY} - \matX \matY} \leq  \eps \norm{\matX} \cdot \norm{\matY}} \geq 1 - \delta,
\end{equation}
\end{theorem}
%
We devote the rest of this section to prove Theorem~\ref{thm:main}. We recall a well-established result from the literature for Gaussian matrices; observe the lack of any upper bound on the error parameter.
\begin{lemma}\label{lem:mm:rank}
Fix $\zeta > 0$ and let $\matX \in \RR^{n \times d}$ and $\matY \in \RR^{d \times m}$. Set $\widehat{\matX} = \matX \matG^\top$ and $\widehat{\matY} = \matG \matY$. Then,
\begin{small}
\begin{align}\nonumber
\Prob{ \norm{\widehat{\matX}\widehat{\matY} - \matX \matY } > \zeta \norm{\matX}\norm{\matY}} \leq c_2^{r(\matX)+r(\matY)} e^{- c_1t \zeta \cdot \min(\zeta, 1)}
\end{align}
\end{small}
where $c_1$ is the Hanson-Wright constant~\cite{ineq:HR} and $c_2=18$.
\end{lemma}
\begin{proof}
The proof is a corollary of the Hanson-Wright inequality~\cite[Theorem 1.1]{ineq:HR}, combined with a dense net argument on the unit sphere defined by the union of the column and row span of $\matY$ and $\matX$, respectively.
\end{proof}

By homogeneity, we observe that, to prove $\norm{\matX \matG^\top \matG \matY - \matX \matY} \leq \varepsilon \norm{\matX} \norm{\matY}$ is satisfied with some probability, it suffices to prove the same argument for $\norm{\frac{\matX}{\norm{\matX}} \matG^\top \matG \frac{\matY}{\norm{\matY}} - \frac{\matX}{\norm{\matX}} \cdot \frac{\matY}{\norm{\matY}}} \leq \varepsilon$. Thus, without loss of generality, we can assume that $\norm{\matX} = \norm{\matY} = 1$.

Let $\matX = \sum_{i=1}^{r(\matX)} \sigma_i \mathbf{u}_i \mathbf{v}_i^\top$ be the singular value decomposition (SVD) of $\matX$, where $\sigma_1, \dots, \sigma_{r(\matX)}$ denote the singular values of $\matX$ and, $\mathbf{u}_1, \dots, \mathbf{u}_{r(\matX)} \in \RR^n$ and $\mathbf{v}_1, \dots, \mathbf{v}_{r(\matX)} \in \RR^d$ denote the left and right singular vectors, respectively. Similarly, we define the SVD of $\matY$ as $\matY = \sum_{j=1}^{r(\matY)} \xi_j \mathbf{p}_j \mathbf{q}_j^\top$.

Define $\theta := \left \lfloor c_3\cdot \frac{\nr{\matX} + \nr{\matY}}{\eps^2}\right\rfloor$ where $c_3 > 1$ is a constant and set $t:= \theta + 8\ln(8/\delta) /\eps^2 + \ln\left( \left\lceil \ln(e/\eps)\right\rceil\right) / (c_1 \cdot \eps^2)$. Given $\theta$, one can decompose $\matX$ as $\matX = \matX_{\theta} + \matX_{\theta}^c$ where $\matX_{\theta}$ represents the \emph{best} rank-$\theta$ approximation of $\matX$ and $\matX_{\theta}^c := \matX - \matX_{\theta}$. Similarly, we can decompose $\matY = \matY_{\theta} + \matY_{\theta}^c$.
Now, by the triangle inequality:
\begin{align}
\|\matX \matG^\top \matG \matY &- \matX \matY\|_2 \leq \norm{ \matX_{\theta}\matG^\top \matG\matY_{\theta} - \matX_{\theta} \matY_{\theta}} \label{eq:result_1}\\
&+ \norm{\matX_{\theta}^c\matG^\top \matG\matY_{\theta}} + \norm{\matX_{\theta}\matG^\top \matG\matY_{\theta}^c} + \norm{\matX_{\theta}^c\matG^\top \matG\matY_{\theta}^c} \label{eq:termNormGaussian}\\
&+ \norm{\matX_{\theta}^c\matY_{\theta}} + \norm{\matX_{\theta}\matY_{\theta}^c} + \norm{\matX_{\theta}^c\matY_{\theta}^c}. \label{eq:termConstant}
\end{align} 

Bounding the term appearing on the right hand side of \eqref{eq:result_1} is the most challenging task and it is the main technical contribution of this paper. Before we start, we need to define some notation. Let $l, s$ be non-negative integers and define the sets $\mathcal{I} = \lbrace \mathcal{I}_1, \mathcal{I}_2, \dots \rbrace$ and $\mathcal{J} = \lbrace \mathcal{J}_1, \mathcal{J}_2, \dots \rbrace$ where:
\begin{align}{\label{eq:def_sets1}} \vspace{-0.1cm}
\mathcal{I}_l := \left\{ i ~ : ~e^{-l} \leq \sigma_i < e^{-l + 1} \right\}, \vspace{-0.1cm}
\end{align} and
\begin{align}{\label{eq:def_sets2}} \vspace{-0.1cm}
\mathcal{J}_s := \left\{ j ~ : ~e^{-s} \leq \xi_j < e^{-s + 1} \right\}. \vspace{-0.1cm}
\end{align} 
We remind that, since $\norm{\matX} = \norm{\matY} = 1$, $\sigma_i \leq 1, ~\xi_j \leq 1$, for any $i, j$.

Given a set $\mathcal{I}_l$, we define $\matX^{l}_{\theta} \in \RR^{n \times d}$ as $\matX^{l}_{\theta} := \sum_{i \in \mathcal{I}_l} \sigma_i \mathbf{u}_i \mathbf{v}_i^\top$; similarly, we have $\matY^{s}_{\theta} := \sum_{j \in \mathcal{J}_s} \xi_j \mathbf{p}_j \mathbf{q}_j^\top$ for the case of $\matY$. We highlight the following key observations:
\begin{itemize}
\item [$(i)$] $\matX_{\theta} = \sum_{l = 1}^{L} \matX^{l}_{\theta}$ and $\matY_{\theta} = \sum_{s = 1}^{S} \matY^{s}_{\theta}$, for some bounded positive integers $L, S$.
\item [$(ii)$] $r(\matX^{l}_{\theta}) = \text{card}(\mathcal{I}_l)$ and $r(\matY^{s}_{\theta}) = \text{card}(\mathcal{J}_s)$.
\item [$(iii)$] $\norm{\matX^{l}_{\theta}} < e^{-l + 1}$ and $\norm{\matY^{s}_{\theta}} < e^{-s + 1}$ by definition of sets $\mathcal{I}_l$ and $\mathcal{J}_s$ respectively in \eqref{eq:def_sets1}-\eqref{eq:def_sets2}.
\end{itemize}
In our analysis, we select $L$ and $S$ such that $\norm{\matX_{\theta}^c} \leq \eps$ and $\norm{\matY_{\theta}^c} \leq \eps$. By definition of $\mathcal{I}_{l}, ~\mathcal{J}_s$, one can deduce that $\norm{\matX_{\theta}^c} \leq e^{L+1}$ and $\norm{\matY_{\theta}^c} \leq e^{S+1}$, for fixed $L$ and $S$. To this end, we require $L = \left \lceil \ln\left(\frac{e}{\eps}\right) \right \rceil$ and $S = \left \lceil \ln\left(\frac{e}{\eps}\right) \right \rceil$ for both $\norm{\matX_{\theta}^c} \leq \eps$ and $\norm{\matY_{\theta}^c} \leq \eps$ to be satisfied.

Based on the definitions above and using triangle inequality on $\norm{\matX_{\theta} \left(\matG^\top \matG - \mat{I}\right)\matY_{\theta}} $, we obtain:
\begin{align}
\norm{\matX_{\theta} \left(\matG^\top \matG - \mat{I}\right)\matY_{\theta}} &\stackrel{(i)}{=} \norm{\left(\sum_{l = 1}^{L} \matX^{l}_{\theta}\right) \left(\matG^\top \matG - \mat{I}\right)\matY_{\theta} } \nonumber \\
&\stackrel{(iv)}{\leq} \sum_{l = 1}^{L} \norm{\matX^{l}_{\theta}\left(\matG^\top \matG - \mat{I}\right)\matY_{\theta}} \nonumber \\ 
&\stackrel{(i)}{=} \sum_{l = 1}^{L}\norm{\matX^{l}_{\theta} \left(\matG^\top \matG - \mat{I}\right)\left(\sum_{s = 1}^{S} \matY^{s}_{\theta}\right)} \nonumber \\
&\stackrel{(iv)}{\leq} \sum_{s = 1}^{S} \left(\sum_{l = 1}^{L} \norm{\matX^{l}_{\theta}\left(\matG^\top \matG - \mat{I}\right)\matY^{s}_{\theta}}\right). \label{eq:result_2}
\end{align} 

For a pair $(l,~s)$ of positive integers, we define the following event, over the probability space defined by $\matG$:
\begin{small}
\begin{align}
&\Omega(l,~s) \nonumber \\ &:= \left\{ \norm{\matX^{l}_{\theta}\left(\matG^\top \matG - \mat{I}\right) \matY^{s}_{\theta}} \geq \left( \eps + c_2\frac{r(\matX^{l}_{\theta}) + r(\matY^{s}_{\theta})}{ c_1 \eps t} \right) \norm{\matX^{l}_{\theta}} \norm{\matY^{s}_{\theta}}\right\}, \nonumber
\end{align}
\end{small} where $c_1,~c_2 > 0$ are positive constants. 
According to Lemma~\ref{lem:mm:rank}, the above event holds with probability:
\begin{align}
\mathbb{P}\left(\Omega(l, s)\right) &\leq c_2^{r(\matX_{\theta}^{l}) + r(\matY_{\theta}^s)}\cdot e^{-c_1 \gamma_{l,s}t \cdot \min(\gamma_{l,s}, 1)} \label{eq:proba1}
\end{align}
where $\gamma_{l,s} = \eps + c_2\frac{r(\matX^{l}_{\theta}) + r(\matY^{s}_{\theta})}{ c_1 \eps t}$. A useful observation for \eqref{eq:proba1} is given in the next lemma; the proof is provided in the appendix.
\begin{lemma}{\label{lem:prob_bound}}
Fix integer $t > 0$. The following inequality holds
\begin{align}{\label{eq:proba3}}
\Prob{\Omega(l, s)} &\leq e^{-c_1 t \cdot \min(\eps, \eps^2)}, \quad \forall l, s.
\end{align}
\end{lemma}

\begin{figure*}[tb]
\begin{align}
\|\matX_{\theta} \big(\matG^\top \matG &- \mat{I}\big)\matY_{\theta}\|_2 \leq  \sum_{s = 1}^{S} \sum_{l = 1}^{L} \left( \left( \eps + c_2\frac{r(\matX^{l}_{\theta}) + r(\matY^{s}_{\theta})}{ c_1 t} \right) \norm{\matX^{l}_{\theta}} \norm{\matY^{s}_{\theta}} \right) \nonumber \\ 
&= \eps \sum_{s = 1}^{S} \sum_{l = 1}^{L} \norm{\matX^{l}_{\theta}} \norm{\matY^{s}_{\theta}}  + \frac{c_2}{c_1 t}\sum_{s = 1}^{S} \sum_{l = 1}^{L} r(\matX^{l}_{\theta})\norm{\matX^{l}_{\theta}} \norm{\matY^{s}_{\theta}} + \frac{c_2}{c_1 t}\sum_{s = 1}^{S} \sum_{l = 1}^{L} r(\matY^{s}_{\theta}) \norm{\matX^{l}_{\theta}} \norm{\matY^{s}_{\theta}} \nonumber \\
&\stackrel{(iii)}{<} \eps\left(\sum_{s = 1}^{S} e^{-s + 1}\right)\left(\sum_{l = 1}^{L} e^{- l + 1}\right) + \frac{c_2}{c_1 t} \left(\sum_{s = 1}^{S} e^{-s + 1}\right) \left(\sum_{l = 1}^{L} r(\matX^{l}_{\theta})\norm{\matX^{l}_{\theta}}\right) + \frac{c_2}{c_1 t} \left(\sum_{l = 1}^{L} e^{-l + 1}\right) \left(\sum_{s = 1}^{S}r(\matY^{s}_{\theta}) \norm{\matY^{s}_{\theta}}\right) \nonumber \\ 
&\leq \eps\left(\sum_{s = 1}^{\infty} e^{-s + 1}\right)\left(\sum_{l = 1}^{\infty} e^{- l + 1}\right) + \frac{c_2}{c_1 t} \left(\sum_{s = 1}^{\infty} e^{-s + 1}\right) \left(\sum_{l = 1}^{L} r(\matX^{l}_{\theta})\norm{\matX^{l}_{\theta}}\right) + \frac{c_2}{c_1 t} \left(\sum_{l = 1}^{\infty} e^{-l + 1}\right) \left(\sum_{s = 1}^{S}r(\matY^{s}_{\theta}) \norm{\matY^{s}_{\theta}}\right) \nonumber \\ 
&= \frac{\eps \cdot e^2}{(e - 1)^2} + \frac{c_2}{c_1 t}\cdot\frac{e}{e - 1} \left[\left(\sum_{l = 1}^{L} r(\matX^{l}_{\theta})\norm{\matX^{l}_{\theta}}\right) + \left(\sum_{s = 1}^{S}r(\matY^{s}_{\theta}) \norm{\matY^{s}_{\theta}}\right)\right] \label{eq:result_3}
\end{align} 
\hrulefill
\end{figure*}
By conditioning on the event $ \bigcap_{s = 1}^{S}\bigcap_{l = 1}^{L} \Omega(l,s)^c $, we can upper bound \eqref{eq:result_2} as in \eqref{eq:result_3} where the infinite series $\sum_{i = 1}^{\infty} e^{ - i+1} = \frac{e}{e- 1}$ is used in the last equality.

To proceed, we observe the following for $\sum_{l = 1}^{L} r(\matX^{l}_{\theta})\norm{\matX^{l}_{\theta}}$; similar reasoning applies for $\sum_{s = 1}^{S}r(\matY^{s}_{\theta}) \norm{\matY^{s}_{\theta}}$. By construction in $(iii)$: 
\begin{align}
\norm{\matX^{l}_{\theta}} < e^{-l+1} = e \cdot e^{-l} \leq e \cdot \min_{i\in \mathcal{I}_l} \sigma_i. \nonumber
\end{align} Thus, it is obvious that:
\begin{align}
r(\matX^{l}_{\theta}) \cdot \norm{\matX^{l}_{\theta}} < e \cdot r(\matX^{l}_{\theta}) \cdot \min_{i\in \mathcal{I}_l} \sigma_i \leq e\cdot \|\matX^{l}_{\theta}\|_{\star}, \nonumber
\end{align} since $\|\matE\|_{\star} \geq  r(\matE) \cdot \min_j \sigma_j(\matE)$, for any matrix $\matE$. To this end,
\begin{align}
\sum_{l = 1}^{L} r(\matX^{l}_{\theta})\norm{\matX^{l}_{\theta}} < e \sum_{l = 1}^{L} \|\matX^{l}_{\theta}\|_{\star} = e \|\matX_{\theta}\|_{\star}.  \nonumber
\end{align} Similarly, we have $\sum_{s = 1}^{S}r(\matY^{s}_{\theta}) \norm{\matY^{s}_{\theta}} < e \|\matY_{\theta}\|_{\star}$ and, therefore, \eqref{eq:result_3} becomes:
\begin{align}
\norm{\matX_{\theta} \left(\matG^\top \matG - \mat{I}\right)\matY_{\theta}} &< \frac{\eps \cdot e^2}{(e - 1)^2} + \frac{c_2}{c_1 t}\cdot\frac{e}{e - 1} \left(\|\matX\|_{\star} + \|\matY\|_{\star}\right) \nonumber \\ 
&= \frac{\eps \cdot e^2}{(e - 1)^2} + \frac{c_2}{c_1 t}\cdot\frac{e}{e - 1} \left(\nr{\matX} + \nr{\matY}\right) \nonumber \\ &\leq 3 \eps \label{eq:result_4}
\end{align} where the last equality is satisfied since $t\geq \theta$ ($c_3\geq \frac{c_2e}{c_1 (e - 1)}$). By the union bound and since $0 < \eps < 1$, \eqref{eq:result_4} is violated with probability:
\begin{align}
\mathbb{P}\left(\left(\bigcap_{s = 1}^{S}\bigcap_{l = 1}^{L} \Omega(l,s)^c \right)^c\right) &\stackrel{\eqref{eq:proba3}}{\leq} \sum_{s = 1}^{S} \sum_{l = 1}^{L} e^{-c_1 \cdot \eps^2 \cdot t} \nonumber \\ &= e^{-c_1 \cdot \eps^2 \cdot t + \ln\left( \left\lceil \ln\left(\frac{e}{\eps}\right)  \right\rceil \right) } \leq \frac{\delta}{2} \nonumber
\end{align} where the last inequality is satisfied since $t \geq \frac{\ln(2/\delta) + \ln\left( \left\lceil \ln(e/\eps)\right\rceil\right)}{c_1 \cdot \eps^2}$.

The terms in \eqref{eq:termConstant} can be bounded as follows:
\begin{align}
\norm{\matX_{\theta}^c \matY_{\theta}} &+ \norm{\matX_{\theta} \matY_{\theta}^c} + \norm{\matX_{\theta}^c \matY_{\theta}^c} \nonumber \\ 
&\leq \norm{\matX_{\theta}^c} \norm{\matY_{\theta}} + \norm{\matX_{\theta}} \norm{\matY_{\theta}^c} + \norm{\matX_{\theta}^c} \norm{\matY_{\theta}^c} \nonumber \\
&\leq 2\eps + \eps^2 \leq 3 \eps \nonumber
\end{align}
where we used the fact that $\norm{\matX_\theta},~\norm{\matY_\theta} \leq 1$ and the inequalities $\norm{\matX_\theta^c} \leq \eps$ (similarly for $\matY_\theta^c$).
To bound the terms in \eqref{eq:termNormGaussian}, we provide the next Lemma; the proof is given in the appendix.
\begin{lemma}\label{lem:tailTerms}
If $t\geq \theta + 8 \ln(8/\delta)$ and $0< \eps <1 $, then, with probability at least $\delta / 2$, the term appearing in \eqref{eq:termNormGaussian} is at most $33\eps$.
\end{lemma}
Applying the union bound on Lemma~\ref{lem:tailTerms} and the complement of $ \bigcap_{s = 1}^{S}\bigcap_{l = 1}^{L} \Omega(l,s)^c $, we conclude that the following inequality holds:
\begin{align}
&\norm{\matX \matG^\top \matG \matY - \matX \matY} \leq 39 \eps, \nonumber
\end{align} with probability at least $1 - \delta$. By rescaling $\eps$ we obtain the required result.
\section{Application to data-driven low-dimensional embedding}
As an application of the result above, we consider the following question:

\vskip.1in
\noindent \textsc{Problem:} {\it 
Given a collection $n$ points in $\mathbb{R}^d$, forming a matrix $\matA \in \RR^{n \times d}$ and an error parameter $\eps>0$, construct efficient and approximately accurate low-dimensional embedding $\matG \in \mathbb{R}^{t \times d}$ such that:}
\begin{align}
\|\matG(\mathbf{a}_i - \mathbf{a}_j)^T\|_2^2 \stackrel{\pm \eps}{\approx} \|\mathbf{a}_i - \mathbf{a}_j\|_2^2,\quad \forall \mathbf{a}_i \in \mathbb{R}^{1 \times d}, ~\mathbf{a}_i \in \text{\texttt{Rows}}(\matA). \nonumber
\end{align}
\noindent {\it The target dimension $t$ is to heavily depend on the input data matrix $\matA$ and be independent of the input dimensions.}
\vskip.1in


Inspired by Theorem~\ref{thm:main}, the following theorem proposes a data-dependent randomized low-dimensional embedding $\matG$ with the following guarantees:
\begin{theorem}\label{eq:main_thm}
Let $\matA \in \mathbb{R}^{n \times d}$. Moreover, assume that $\matA$ with nuclear rank $ \nr{\matA} \ll O(\log n)$. If $t = \Omega\left( \frac{\nr{\matA} + \log( \log(1/\eps)) + \log(1/\delta)}{\eps^2}\right)$, then the following inequalities hold 
\begin{align}
\forall i, j \in [n]\quad \left| \left\|\matG\left(\mathbf{a}_i - \mathbf{a}_j\right)^\top\right\|_2^2 - \left\|\mathbf{a}_i - \mathbf{a}_j\right\|_2^2 \right| \leq 2\eps \norm{\matA}^2, \nonumber
\end{align}
with probability at least $1-\delta$.
\end{theorem} 

To fully specify the algorithmic procedure followed in practice, only $\matA \in \mathbb{R}^{n \times d}$ and constants $\delta,~\varepsilon \in (0, 1)$ are given as input and $\matG \in \mathbb{R}^{t \times d}$ is returned. Moreover, $\matG$ guarantees to preserve the distances between the rows of $\matA$ with additive error $2\varepsilon \norm{\matA}^2$ and with probability at least $1- \delta$. The exact steps followed are given in Algorithm \ref{algo:1}.

\begin{algorithm}[!htp]
\caption{Data-Driven Random Linear Embedding}\label{algo:1}
\begin{small}
\begin{algorithmic}[1]
\Statex ~~~\textbf{Input:} $\matA \in \mathbb{R}^{n \times d}$, $\varepsilon, \delta \in (0, 1)$
\algrule
\State ~~~ Compute (or approximate) $\|\matA\|_{\star}$ and $\norm{\matA}$.
\State ~~~ Set parameter $t = \Omega\left( \frac{\nr{\matA} + \log \log(1/\eps) + \log(1/\delta)}{\eps^2} \right)$.
\State ~~~ Generate $\matG \in \mathbb{R}^{t \times n}$ according to Theorem \ref{eq:main_thm}.
\algrule
\State ~~~\textbf{Output:} Matrix $\matG$ and $\matG\matA$ that satisfies Theorem \ref{eq:main_thm} with probability at least $1 - \delta$.
\end{algorithmic}
\end{small}
\end{algorithm}

\begin{proof}
By substituting $\matX = \matY^\top = \matA$ in Theorem \ref{thm:main}, we have:
\begin{align}
\mathbb{P}\left( \norm{\matA \matG^\top \matG \matA^\top - \matA \matA^\top} \leq \eps \norm{\matA}^2\right) \geq 1 - \delta. \nonumber
\end{align}
By the definition of the spectral norm, the condition above can be further written as the following maximization problem:
\begin{align}
\max_{\mathbf{x}: \mathbf{x} \in \mathbb{R}^n, \|\mathbf{x}\|_2 = 1} \left\{ \left| \mathbf{x}^\top \left( \matA \matG^\top \matG \matA^\top - \matA \matA^\top\right) \mathbf{x} \right| \right\} \leq \eps \norm{\matA}^2. \nonumber
\end{align} Moreover, it is obvious that, if we restrict the search space $\mathbf{x} \in \mathbb{R}^n$ to $\mathbf{x} \in \mathbb{B}^n$ where:
\begin{align}
\mathbb{B}^n = \left\{ \mathbf{y} ~ | ~\mathbf{y} = \frac{1}{\sqrt{2}}\left(\mathbf{e}_i - \mathbf{e}_j\right), \forall i, j \in [n] \right\}, \nonumber
\end{align} where $\mathbf{e}_i$ is the standard basis vector with 1 in the $i$-th position, we further have:
\begin{align}
\max_{\mathbf{x} : \mathbf{x} \in \mathbb{B}^n} &\left\{ \left| \mathbf{x}^\top \left( \matA \matG^\top \matG \matA^\top - \matA \matA^\top\right) \mathbf{x} \right| \right\} \nonumber \\ &\leq \max_{\mathbf{x}: \mathbf{x} \in \mathbb{R}^n, \|\mathbf{x}\|_2 = 1} \left\{ \left| \mathbf{x}^\top \left( \matA \matG^\top \matG \matA^\top - \matA \matA^\top\right) \mathbf{x} \right| \right\} \nonumber
\end{align} which leads to:
\begin{align}{\label{eq:2}}
\mathbb{P}\left(\max_{\mathbf{x} : \mathbf{x} \in \mathbb{B}^n} \left\{ \left| \mathbf{x}^\top \left( \matA \matG^\top \matG \matA^\top - \matA \matA^\top\right) \mathbf{x} \right| \right\} \leq \eps \norm{\matA}^2\right) \geq 1 - \delta.
\end{align} 
Observe also that $\left| \mathbf{x}^\top \left( \matA \matG^\top \matG \matA^\top - \matA \matA^\top\right) \mathbf{x} \right| = \left| \|\matG \matA^\top\mathbf{x}\|_2^2 - \|\matA^\top \mathbf{x}\|_2^2 \right|$, which further transforms \eqref{eq:2} as:
\begin{align}{\label{eq:3}}
\mathbb{P}\left(\max_{\mathbf{x} : \mathbf{x} \in \mathbb{B}^n} \left\{ \left| \|\matG \matA^\top\mathbf{x}\|_2^2 - \|\matA^\top \mathbf{x}\|_2^2 \right| \right\} \leq \eps \|\matA\|_2^2\right) \geq 1 - \delta.
\end{align} For any vector $\mathbf{x} \in \mathbb{B}^n$, we observe that $\mathbf{x}^\top \matA \matG^\top = 1/\sqrt{2}\left(\mathbf{a}_i - \mathbf{a}_j\right)\matG^\top$ while $\matG \matA^\top \mathbf{x} = 1/\sqrt{2}\matG \left(\mathbf{a}_i - \mathbf{a}_j\right)^\top$. Similarly, $\mathbf{x}^\top \matA= 1/\sqrt{2}\left(\mathbf{a}_i - \mathbf{a}_j\right)$ and $\matA^\top \mathbf{x} = 1/\sqrt{2}\left(\mathbf{a}_i - \mathbf{a}_j\right)^\top$. Thus, \eqref{eq:3} becomes:
\begin{small}
\begin{align}{\label{eq:4}}
\mathbb{P}\Big(\max_{i, j \in [d], i \neq j} \left\{ \left|~\left\|\matG\left(\mathbf{a}_i - \mathbf{a}_j\right)^\top\right\|_2^2 - \left\|\mathbf{a}_i - \mathbf{a}_j\right\|_2^2~\right| \right\} &\leq 2 \eps \|\matA\|_2^2\Big) \nonumber \\ &\geq 1 - \delta.
\end{align}
\end{small} Since \eqref{eq:4} is satisfied for the maximizing combination of $i, j$, we can safely remove the maximization to get:
\begin{align}
\mathbb{P}\left(\left|~\left\|\matG\left(\mathbf{a}_i - \mathbf{a}_j\right)^\top\right\|_2^2 - \left\|\mathbf{a}_i - \mathbf{a}_j\right\|_2^2~\right| \leq 2\eps \|\matA\|_2^2\right) \geq 1 - \delta, \nonumber
\end{align} which completes the proof.
\end{proof} A complete set of experiments will be included in an extended version of the paper.
\section{Conclusions}

We present a novel analysis for a class of randomized and provably $\varepsilon$-accurate algorithms for the problem of matrix multiplication. 
As an application of this result, we show the utilization of the proposed scheme on data-driven low dimensional embeddings with additive error approximation, in the case where the data live on a subspace characterized by a small \emph{nuclear rank}. 

An interesting question to pursue lies in the substitution of nuclear rank by stable rank: recent developments on this topic \cite{zouzias2013randomized} show similar results using the latter metric as the intrinsic data dimension; a weaker assumption than the nuclear rank. However, the dependence on the approximation error is of the order $O(\frac{1}{\varepsilon^4})$, as opposed to $O(\frac{1}{\varepsilon^2})$ presented in this work. We hope this paper triggers future efforts to improve stable rank-based bounds with respect to error dependency.

\section*{Acknowledgment}
A. Zouzias would like to thank Mark Rudelson for several discussions on the approximate matrix multiplication problem. This research has received funding from the ERC under the EU's Seventh Framework Programme (FP7/2007-2013) / ERC grant agreement $n^{o}$ 259569.

\bibliographystyle{IEEEtran}
\bibliography{randproj}

\begin{thebibliography}{10}
\providecommand{\url}[1]{#1}
\csname url@samestyle\endcsname
\providecommand{\newblock}{\relax}
\providecommand{\bibinfo}[2]{#2}
\providecommand{\BIBentrySTDinterwordspacing}{\spaceskip=0pt\relax}
\providecommand{\BIBentryALTinterwordstretchfactor}{4}
\providecommand{\BIBentryALTinterwordspacing}{\spaceskip=\fontdimen2\font plus
\BIBentryALTinterwordstretchfactor\fontdimen3\font minus
  \fontdimen4\font\relax}
\providecommand{\BIBforeignlanguage}[2]{{%
\expandafter\ifx\csname l@#1\endcsname\relax
\typeout{** WARNING: IEEEtran.bst: No hyphenation pattern has been}%
\typeout{** loaded for the language `#1'. Using the pattern for}%
\typeout{** the default language instead.}%
\else
\language=\csname l@#1\endcsname
\fi
#2}}
\providecommand{\BIBdecl}{\relax}
\BIBdecl

\bibitem{sarlos2006improved}
T.~Sarlos, ``Improved approximation algorithms for large matrices via random
  projections,'' in \emph{Foundations of Computer Science, 2006. FOCS'06. 47th
  Annual IEEE Symposium on}.\hskip 1em plus 0.5em minus 0.4em\relax IEEE, 2006,
  pp. 143--152.

\bibitem{chandrasekaran2013computational}
V.~Chandrasekaran and M.~I. Jordan, ``Computational and statistical tradeoffs
  via convex relaxation,'' \emph{Proceedings of the National Academy of
  Sciences}, vol. 110, no.~13, pp. E1181--E1190, 2013.

\bibitem{achlioptas2013near}
D.~Achlioptas, Z.~Karnin, and E.~Liberty, ``Near-optimal entrywise sampling for
  data matrices,'' \emph{arXiv preprint arXiv:1311.4643}, 2013.

\bibitem{mahoney2009cur}
M.~W. Mahoney and P.~Drineas, ``Cur matrix decompositions for improved data
  analysis,'' \emph{Proceedings of the National Academy of Sciences}, vol. 106,
  no.~3, pp. 697--702, 2009.

\bibitem{BoutsidisMD09}
C.~Boutsidis, M.~W. Mahoney, and P.~Drineas, ``An improved approximation
  algorithm for the column subset selection problem,'' in \emph{SODA}, 2009,
  pp. 968--977.

\bibitem{indyk1998approximate}
P.~Indyk and R.~Motwani, ``Approximate nearest neighbors: towards removing the
  curse of dimensionality,'' in \emph{Proceedings of the thirtieth annual ACM
  symposium on Theory of computing}.\hskip 1em plus 0.5em minus 0.4em\relax
  ACM, 1998, pp. 604--613.

\bibitem{BoutsidisZD10}
C.~Boutsidis, A.~Zouzias, and P.~Drineas, ``Random projections for \$k\$-means
  clustering,'' in \emph{NIPS}, 2010, pp. 298--306.

\bibitem{zouzias2013randomized}
A.~Zouzias, ``Randomized primitives for linear algebra and applications,''
  Ph.D. dissertation, University of Toronto, 2013.

\bibitem{MZ11}
A.~Magen and A.~Zouzias, ``Low rank matrix-valued chernoff bounds and
  approximate matrix multiplication,'' in \emph{SODA}, 2011, pp. 1422--1436.

\bibitem{johnson1984extensions}
W.~B. Johnson and J.~Lindenstrauss, ``Extensions of lipschitz mappings into a
  hilbert space,'' \emph{Contemporary mathematics}, vol.~26, no. 189-206, p.~1,
  1984.

\bibitem{baraniuk}
\BIBentryALTinterwordspacing
R.~G. Baraniuk and M.~B. Wakin, ``\BIBforeignlanguage{English}{Random
  projections of smooth manifolds},''
  \emph{\BIBforeignlanguage{English}{Foundations of Computational
  Mathematics}}, vol.~9, no.~1, pp. 51--77, 2009. [Online]. Available:
  \url{http://dx.doi.org/10.1007/s10208-007-9011-z}
\BIBentrySTDinterwordspacing

\bibitem{clarkson2008tighter}
K.~L. Clarkson, ``Tighter bounds for random projections of manifolds,'' in
  \emph{Proceedings of the twenty-fourth annual symposium on Computational
  geometry}.\hskip 1em plus 0.5em minus 0.4em\relax ACM, 2008, pp. 39--48.

\bibitem{mendelson}
\BIBentryALTinterwordspacing
S.~Mendelson, A.~Pajor, and N.~Tomczak-Jaegermann,
  ``\BIBforeignlanguage{English}{Reconstruction and subgaussian operators in
  asymptotic geometric analysis},''
  \emph{\BIBforeignlanguage{English}{Geometric and Functional Analysis}},
  vol.~17, no.~4, pp. 1248--1282, 2007. [Online]. Available:
  \url{http://dx.doi.org/10.1007/s00039-007-0618-7}
\BIBentrySTDinterwordspacing

\bibitem{Klartag}
\BIBentryALTinterwordspacing
B.~Klartag and S.~Mendelson, ``Empirical processes and random projections,''
  \emph{Journal of Functional Analysis}, vol. 225, no.~1, pp. 229 -- 245, 2005.
  [Online]. Available:
  \url{http://www.sciencedirect.com/science/article/pii/S0022123604003635}
\BIBentrySTDinterwordspacing

\bibitem{related:movingPoints}
\BIBentryALTinterwordspacing
P.~K. Agarwal, S.~Har-Peled, and H.~Yu, ``Embeddings of surfaces, curves, and
  moving points in euclidean space,'' in \emph{Proceedings of the Twenty-third
  Annual Symposium on Computational Geometry}, ser. SCG '07.\hskip 1em plus
  0.5em minus 0.4em\relax New York, NY, USA: ACM, 2007, pp. 381--389. [Online].
  Available: \url{http://doi.acm.org/10.1145/1247069.1247135}
\BIBentrySTDinterwordspacing

\bibitem{JL:beyond}
\BIBentryALTinterwordspacing
Y.~Bartal, B.~Recht, and L.~J. Schulman, ``Dimensionality reduction: Beyond the
  johnson-lindenstrauss bound,'' in \emph{Proceedings of the Twenty-Second
  Annual ACM-SIAM Symposium on Discrete Algorithms}, ser. SODA '11.\hskip 1em
  plus 0.5em minus 0.4em\relax SIAM, 2011, pp. 868--887. [Online]. Available:
  \url{http://dl.acm.org/citation.cfm?id=2133036.2133104}
\BIBentrySTDinterwordspacing

\bibitem{JL:beyond2}
\BIBentryALTinterwordspacing
L.-A. Gottlieb and R.~Krauthgamer, ``A nonlinear approach to dimension
  reduction,'' in \emph{Proceedings of the Twenty-Second Annual ACM-SIAM
  Symposium on Discrete Algorithms}, ser. SODA '11.\hskip 1em plus 0.5em minus
  0.4em\relax SIAM, 2011, pp. 888--899. [Online]. Available:
  \url{http://dl.acm.org/citation.cfm?id=2133036.2133105}
\BIBentrySTDinterwordspacing

\bibitem{nnEmbed}
\BIBentryALTinterwordspacing
P.~Indyk and A.~Naor, ``Nearest-neighbor-preserving embeddings,'' \emph{ACM
  Trans. Algorithms}, vol.~3, no.~3, Aug. 2007. [Online]. Available:
  \url{http://doi.acm.org/10.1145/1273340.1273347}
\BIBentrySTDinterwordspacing

\bibitem{hegde2012convex}
C.~Hegde, A.~Sankaranarayanan, W.~Yin, and R.~Baraniuk, ``A convex approach for
  learning near-isometric linear embeddings,'' \emph{preparation, August},
  2012.

\bibitem{grant2013nearly}
E.~Grant, C.~Hegde, and P.~Indyk, ``Nearly optimal linear embeddings into very
  low dimensions,'' \emph{IEEE GlobalSIP Symposium on Sensing and Statistical
  Inference, Austin, TX}, 2013.

\bibitem{ineq:HR}
M.~Rudelson and R.~Vershynin, ``Hanson-wright inequality and sub-gaussian
  concentration,'' \emph{Electron. Commun. Probab.}, vol.~18, pp. no. 82, 1--9,
  2013.

\end{thebibliography}

\section*{Appendix}

In our analysis, we make use of the following concentration bound on the operator norm of the product of a fixed matrix with a Gaussian matrix which is a direct consequence of concentration of Lipschitz function on Gaussian space, see e.g.~\cite[p.~10]{zouzias2013randomized}.
\begin{lemma}\label{lem:normGaussian}
Let $\matX \in \mathbb{R}^{n\times d}$ and $t\geq 1$. For every $\tau>0$:
\begin{equation}
\Prob{ \norm{\matX\matG^\top} \geq \frobnorm{\matX} / \sqrt{t} + \norm{\matX} + \tau \norm{\matX}/\sqrt{t}} \leq \exp(-\tau^2 /8). \nonumber
\end{equation}
\end{lemma}

\subsection{Proof of Lemma~\ref{lem:prob_bound}}
In all the cases below, observe $\gamma_{l,s} > \eps$ by definition. Using Eqn.~\eqref{eq:proba3}, we consider the following two cases:
\begin{itemize}
\item [$(i)$] $\gamma_{l,s}  < 1$: in this case, we have $\min(\gamma_{l,s}, 1) \geq \min(\eps, 1) = \eps$. Thus, \eqref{eq:proba3} becomes:
\begin{align}
\mathbb{P}\left(\Omega(l, s)\right) &\leq c_2^{r(\matX_{\theta}^{l}) + r(\matY_{\theta}^s)}\cdot e^{-c_1\left(\eps + c_2\frac{r(\matX^{l}_{\theta}) + r(\matY^{s}_{\theta})}{ c_1 \eps t}\right)t \cdot \eps} \nonumber \\ 
&\leq e^{-c_1\eps^2 t} \cdot e^{-c_2\left(r(\matX^{l}_{\theta}) + r(\matY^{s}_{\theta})\right)} \cdot e^{\ln (c_2)\left(r(\matX^{l}_{\theta}) + r(\matY^{s}_{\theta})\right)} \nonumber \\
&\leq e^{-c_1\eps^2 t}
\end{align}
\item [$(ii)$] $\gamma_{l,s}  \geq  1$: in this case, $\min(\gamma_{l,s}, 1) = 1$. Following the same steps as above, we have:
\begin{align}
\mathbb{P}\left(\Omega(l, s)\right) &\leq c_2^{r(\matX_{\theta}^{l}) + r(\matY_{\theta}^s)}\cdot e^{-c_1\left(\eps + c_2\frac{r(\matX^{l}_{\theta}) + r(\matY^{s}_{\theta})}{ c_1 \eps t}\right)t} \nonumber \\ 
&\leq e^{-c_1\eps t} \cdot e^{\left(1 - \frac{1}{\eps}\right)c_2\left(r(\matX^{l}_{\theta}) + r(\matY^{s}_{\theta})\right)}.  \nonumber
\end{align}
However, since $0 < \eps < 1$, we have $1 - \frac{1}{\eps} < 0$ and thus, the above inequality can be further upper bounded by:
\begin{align}
\mathbb{P}\left(\Omega(l, s)\right) \leq e^{-c_1\eps t}. \nonumber
\end{align}
\end{itemize}
\subsection{Proof of Lemma~\ref{lem:tailTerms}}
Apply Lemma~\ref{lem:normGaussian} with $\tau = \sqrt{t}$ on all four matrices $\matX_{\theta}^c$, $\matX_{\theta}$, $\matY_{\theta}$ and $\matY_{\theta}^c$. With probability at least $1-\exp(-t/8)$:
\begin{align}
		\norm{\matX_{\theta}^c\matG^\top} &\leq \frac{\frobnorm{\matX_{\theta}^c}}{\sqrt{t}}  + \left(1 + \frac{\tau}{\sqrt{t}}\right)\norm{\matX_{\theta}^c} \nonumber \\ &\stackrel{\tau = \sqrt{t}}{=} \frac{1}{\sqrt{t}} \frobnorm{\matX_{\theta}^c} + 2\eps \nonumber \\
										&\leq \frac{1}{\sqrt{t}} \frobnorm{\matX} + 2\eps \leq \frac{1}{\sqrt{t}} \cdot \|\matX\|_{\star} + 2\eps  \nonumber\\
									   &\stackrel{t \geq \theta}{\leq} \frac{1}{\sqrt{\theta}} \nr{\matX} + 2\eps \leq \frac{\eps}{\sqrt{c_3}} + 2\eps \stackrel{c_3 > 1}{\leq} 3\eps \nonumber
\end{align} 
by definition of $\theta$ and $\norm{\matX_{\theta}^c} \leq \eps$.
Similarly, with probability at least $1-\exp(-t/8)$:
\begin{align}
\norm{\matX_{\theta}\matG^\top} &\leq  \frac{\frobnorm{\matX_{\theta}}}{\sqrt{t}} + \left(1 + \frac{\tau}{\sqrt{t}}\right)\norm{\matX_{\theta}} \nonumber \\ &\stackrel{\tau = \sqrt{t}}{=} \frac{1}{\sqrt{t}} \frobnorm{\matX_{\theta}} + 2 \nonumber \\
								&\leq \frac{1}{\sqrt{t}} \frobnorm{\matX} + 2 \leq \frac{1}{\sqrt{t}} \cdot \|\matX\|_{\star} + 2  \nonumber\\
	     					   &\stackrel{t \geq \theta}{\leq}  \frac{1}{\sqrt{\theta}} \nr{\matX} + 2\nonumber \leq \frac{\eps}{\sqrt{c_3}} + 2 \stackrel{c_3 > 1}{\leq} 2(\eps + 1) \nonumber
\end{align}
		by definition of $\theta$ and $\norm{\matX_{\theta}} = 1$. 
We further observe $\norm{\matY_{\theta}^c\matG^\top} \leq 3\eps$ and $\norm{\matX_{\theta}\matG^\top} \leq 2(\eps + 1)$ both with probability at least $1-\exp(-t/8)$. 
	Union bound all the above four results, it follows that
	\begin{align}
	&\norm{ \matX_{\theta}^c\matG^\top \matG\matY_{\theta}} + \norm{\matX_{\theta}\matG^\top \matG\matY_{\theta}^c} + \norm{ \matX_{\theta}^c\matG^\top \matG\matY_{\theta}^c}  \nonumber \\ 
	&\leq 12\eps\left(\eps + 1\right) + 9\eps^2 \leq 12\eps  + 21\eps^2 \leq 33\eps \nonumber
	\end{align} 
which holds with probability at most $4\cdot e^{-t/8}$.  By definition of $t$, the above inequality is violated with probability at most $\delta/2$.

\end{document}